\documentclass[a4paper,12pt]{article}
\usepackage{graphics}
\usepackage{epsf}

\usepackage{amsfonts}
\usepackage{amssymb}
\usepackage{epsfig}
\usepackage[latin1]{inputenc}
\usepackage{color}

\usepackage{amsthm,amsmath,amssymb}

\newtheorem{theorem}{Theorem}

\begin{document}

%\title{Gautschi method without order reduction when integrating  boundary value nonlinear wave problems}
%%\author{}
%\author{
%{\sc  B. Cano \thanks{Corresponding author. Email:  bego@mac.uva.es} } \\ \small
%IMUVA, Departamento de Matem\'atica Aplicada,\\ \small Facultad de Ciencias, Universidad de
%Valladolid,\\ \small Paseo de Bel\'en 7, 47011 Valladolid,\\ \small Spain \\
%{\sc and}\\
%{\sc M. J. Moreta}\thanks{Email: mjesusmoreta@ccee.ucm.es} \\
%\small IMUVA,
%Departamento de Fundamentos del Análisis Económico I, \\ \small
%Facultad de Ciencias Económicas y Empresariales, Universidad
%Complutense de Madrid, \\[2pt] \small Campus de Somosaguas, Pozuelo de Alarcón,
%28223 Madrid, \\ \small Spain.
%}
%\date{}

%\maketitle

\title{Looking for efficiency when avoiding order reduction\\
  in nonlinear problems with Strang splitting}

\author{{\sc I. Alonso-Mallo \thanks{Email: isaias@mac.uva.es}}, 
{\sc B. Cano \thanks{Email: bego@mac.uva.es}}
\\ \small
IMUVA, Departamento de Matem\'atica Aplicada ,\\ \small Facultad de Ciencias, Universidad de
Valladolid,\\ \small Paseo de Bel\'en 7, 47011 Valladolid,\\ \small Spain \\
{\sc N. Reguera \thanks{Email: nreguera@ubu.es}} \\
\small IMUVA,
Departamento de Matem\'aticas y Computaci\'on, \\ \small
Escuela Polit\'ecnica Superior, Universidad
de Burgos, \\ \small Avda. Cantabria, 09006 Burgos, \\ \small Spain.}

\date{}
\maketitle

%\date{\today}

\begin{abstract}
In this paper, we offer a comparison in terms of computational efficiency between two techniques to avoid order reduction when using Strang method to integrate nonlinear initial boundary value problems with time-dependent boundary conditions. Considering different implementations for each of the techniques, we show that the technique suggested by Alonso et al. is more efficient than the one suggested by Einkemmer et al. Moreover, for one of the implementations of the technique by Alonso et al. we justify its order through the proof of some new theorems.
\end{abstract}


%
%  Main text of the article.
%

\section{Introduction}
There are several papers in the literature concerning the important fact of avoiding the order reduction in time which turns up when integrating with splitting methods nonlinear problems
of the form
\begin{eqnarray}
u'(t) &=& A u(t)+f(t,u(t)), \quad 0 \le t \le T, \nonumber \\
\partial u(t)&=&g(t), \nonumber \\
u(0)&=&u_0, \label{ibvp}
\end{eqnarray}
where $A$ is an elliptic differential operator, $f$ is a smooth real function which acts as a reaction term, $\partial$ is a boundary operator, $g$ is the boundary condition which in principle does not vanish and is time-dependent and $u_0$ is a smooth initial condition which makes that the solution of  (\ref{ibvp}) is regular enough.

 More particularly, in \cite{EO1,EO2} a technique is suggested to do it, in which each part of the splitting is assumed to be solved in an exact way for the analysis and, in the numerical experiments, standard subroutines are used to integrate each part in space and time. Although, from the point of view of the analysis, the technique in both papers is equivalent, the difference is that, in \cite{EO1}, the solution of this elliptic problem is required at each time $t\in[0,T]$,
 $$Az(t)=0, \quad \partial z(t)=g(t),$$
 and a suggestion for $z_t(t)$ must be given. This can be very simple analytically in one dimension, but it is much more complicated and expensive in several dimensions. Nevertheless, that is avoided in \cite{EO2} by considering  just a function $q$ which coincides with $f(g)$ at the boundary. That can also be done analytically in one dimension and simple domains in two dimensions, and numerically in more complicated domains, according to a remark made in \cite{EO2} although it is not in fact applied to such a problem there. In this paper, we will concentrate on the technique in \cite{EO2} for 1-dimensional and 2-dimensional simple domains.

On the other hand, in \cite{ACRnl}, another different technique is suggested in which appropriate boundary conditions are suggested for each part of the splitting. The analysis there considers both the space and time discretization. The linear and stiff part is integrated \lq exactly' in time through exponential-type functions while the nonlinear but smooth part is assumed to be numerically integrated by a classical integrator just with the order of accuracy that the user wants to achieve with the whole method. Although the latter seems to be the most natural, in order to be more similar in the comparison with the technique in \cite{EO2}, we will use standard subroutines which use variable stepsizes with given small tolerances for the nonlinear and smooth problems of both techniques.

We will concentrate on the extensively used second-order Strang splitting and the aim of the paper is to compare both techniques in terms of computational efficiency, considering different space discretizations, different tolerances for the standard subroutines which integrate in time some of the split problems, and different (although standard) ways to tackle the calculation of terms which contain exponential-type functions of matrices. For that, both one-dimensional and bidimensional problems will be considered. There is already another comparison in the literature between both techniques \cite{EO3} but there they just compare in terms of error against the time stepsize without entering into the details of implementation and its computational cost, which we believe that is the interesting comparison. Moreover, they just consider time-independent boundary conditions and $1$-dimensional problems, for which many simplifications can be made.

The paper is structured as follows.  Section 2 gives some preliminaries on the description of the different techniques and suggest different implementations for each of them.
Section 3 presents results for all the suggested methods in terms of error against cpu time when using accurate spectral collocation methods in space and very small tolerances for the
standard subroutines in time. Section 4 offers also a numerical comparison, but now using less accurate finite differences in space and less small tolerances for the standard
subroutines in time. Moreover, numerical differentiation is also considered in order to achieve local order 3 instead of just 2, and again the computational comparison is performed.
Finally, in an appendix, a thorough error analysis (including numerical differentiation) is given for one of the implementations of the main technique which was suggested in \cite{ACRnl}, but which modifications were not
included in the analysis.

\section{Preliminaries and suggestion of different implementations}

The technique which is suggested in \cite{EO2} consists of the following: A function $q(t)$ is constructed which satisfies $\partial q(t)=\partial f(t,u(t))$. Then, given the numerical approximation at the previous step $u_n$, the numerical approximation at the next step $u_{n+1}$ is given by the following procedure:
\begin{eqnarray}
&&\left\lbrace\begin{array}{rcl} v_{n,1}'(t)&=& A v_{n,1}(t)+q(t),  \\
v_{n,1}(t_n)&=&u_n, \\
\partial v_{n,1}(t)&=&g(t), \end{array}\right. \nonumber \\
&&\left\lbrace\begin{array}{rcl}  w_n'(t)&=&f(t,w_n(t))-q(t),  \\
w_n(t_n)&=& v_{n,1}(t_n+\frac{k}{2}), \end{array} \right.\nonumber \\
&&\left\lbrace\begin{array}{rcl} v_{n,2}'(t)&=& A v_{n,2}(t)+q(t),  \\
v_{n,2}(t_n+\frac{k}{2})&=&w_n(t_n+k), \\
\partial v_{n,2}(t)&=&g(t), \end{array}\right. \nonumber \\
&&u_{n+1}=v_{n,2}(t_n+k). \label{impl2eo}
\end{eqnarray}
However, we notice that two of three problems which turn up here are stiff and therefore solving them will be more expensive than solving the unique nonlinear but smooth problem. In order to reverse that, the decomposition of the splitting method can be done in another order and then the following procedure would turn up, for which with similar arguments, no order reduction would either turn up:
\begin{eqnarray}
&&\left\lbrace\begin{array}{rcl} w_{n,1}'(t)&=& f(t,w_{n,1}(t))-q(t), \\
w_{n,1}(t_n)&=&u_n, \end{array} \right.\nonumber \\
&&\left\lbrace\begin{array}{rcl}  v_n'(t)&=& A v_n(t)+q(t),  \\
v_n(t_n)&=& w_{n,1}(t_n+\frac{k}{2}), \\
\partial v_n(t)&=&g(t)\end{array} \right.\nonumber \\
&&\left\lbrace\begin{array}{rcl} w_{n,2}'(t)&=& f(t,w_{n,2}(t))-q(t),  \\
w_{n,2}(t_n+\frac{k}{2})&=&v_n(k), \end{array}\right. \nonumber \\
&&u_{n+1}=w_{n,2}(t_n+k). \label{impl1eo}
\end{eqnarray}
Then, two of the problems are cheap and just one is more expensive.

On the other hand, in \cite{ACRnl}, the main idea is to consider, from $u_n$,
\begin{eqnarray}
&&\left\lbrace\begin{array}{rcl}
w_{n}'(s)&=&A w_{n}(s),\\
w_{n}(0)&=&\Psi_{\frac{k}{2}}^{f,t_n}(u_n),\\
\partial w_{n}(s)&=&\partial [ u(t_n)+\frac{k}{2}f(t_n,u(t_n))+s A u(t_n)],
\end{array} \right.\nonumber \\
&& u_{n+1}=\Psi_{\frac{k}{2}}^{f,t_n+\frac{k}{2}}(w_n(k)). \nonumber
\end{eqnarray}
where $\Psi_{\frac{k}{2}}^{f,t_n}$ and $\Psi_{\frac{k}{2}}^{f,t_n+\frac{k}{2}}$ integrate respectively with order $2$ the following problems from $s=0$ to $s=k/2$:
$$
v'_{n}(s)=f(t_n+s, v_n(s)), \quad z_n'(s)=f(t_n+\frac{k}{2}+s, z_n(s)).$$
Moreover, the procedure to integrate this is more explicitly stated. Firstly, in \cite{ACRnl}
(see also \cite{ACRl,ACRLaw,CR}), a general space discretization is introduced which discretizes the elliptic problem
$$A u=F, \quad \partial u=g,$$
through the \lq elliptic projection' $R_h u$ which satisfies
$$A_{h,0} R_h u+ C_h g= P_h F,$$
for a certain matrix $A_{h,0}$, an associated boundary operator $C_h$ and a projection operator $P_h$. Then, given the numerical approximation at the previous step $U_h^n$, the procedure in \cite{ACRnl} to obtain $U_h^{n+1}$ reads as follows:
\begin{eqnarray}
V_h^n &=& \Psi_{\frac{k}{2}}^{f,t_n}(U_h^n), \nonumber \\
W_{h,n}(k)&=&e^{k A_{h,0}} V_{h}^n+ k \varphi_1( k A_{h,0})C_h [g(t_n)+\frac{k}{2} \partial
       f(t_n,u(t_n))]
\nonumber \\&& + k^2 \varphi_2( k A_{h,0}) C_h [g'(t_n)-\partial f(t_n,u(t_n)] \nonumber \\
U_h^{n+1}&=&\Psi_{\frac{k}{2}}^{f,t_n+\frac{k}{2}}(W_{h,n}(k)), \label{impl1acr}
\end{eqnarray}
where $\varphi_1$ and $\varphi_2$ are the standard functions which are used in exponential methods \cite{ACRnl}.
The original suggestion used this order for the decomposition thinking that $\Psi_k$ is just an explicit method which is applied with a single stepsize $k$, and therefore it would be cheaper than the equation in $W_{h,n}(k)$. We still believe that would be the best. However, as in this paper, in order to do it more similarly to \cite{EO2}, we will solve that part with a standard variable stepsize subroutine for non-stiff problems until a given small tolerance, the first and last problem may be more expensive than the middle one. Therefore, we will also consider this other implementation which comes from reversing the order of the problems in the decomposition (see the appendix):
\begin{eqnarray}
W_{h,n}(\frac{k}{2})&=&e^{\frac{k}{2} A_{h,0}} U_{h}^n+ \frac{k}{2} \varphi_1( \frac{k}{2} A_{h,0})C_h g(t_n)+\frac{k^2}{4} \varphi_2( \frac{k}{2} A_{h,0})C_h\partial
       Au(t_n)
\nonumber \\
V_h^n &=& \Psi_{k}^{f,t_n}(W_{h,n}(\frac{k}{2})), \nonumber \\
U_h^{n+1}&=&e^{\frac{k}{2} A_{h,0}}V_{h}^n + \frac{k}{2} \varphi_1( \frac{k}{2} A_{h,0})C_h \partial [u(t_n)+\frac{k}{2}A u(t_n)+ k f(t_n,u(t_n)] \nonumber \\
&& + \frac{k^2}{4} \varphi_2( \frac{k}{2} A_{h,0}) C_h \partial A u(t_n).
\label{impl2acr}
\end{eqnarray}
In the following, we will denote by {\bf EO1} to (\ref{impl1eo}), by {\bf EO2} to (\ref{impl2eo}), by {\bf ACR1} to (\ref{impl1acr}) and by {\bf ACR2} to (\ref{impl2acr}).

\section{Numerical comparison with spectral collocation methods in space and high accuracy in time}

In this section, a comparison in terms of computational efficiency among the different techniques is given when solving all the problems at hand with high accuracy.
In such a way, we will be seeing the error which comes from the splitting itself for a large range of values of the timestepsize $k$.
More precisely, spectral collocation methods  \cite{BM} are used for the space discretization with the two implementations of both techniques. Moreover, the nonlinear and non-stiff problems of both techniques are integrated with MATLAB subroutine ode45
with relative tolerance $10^{-12}$ and absolute tolerance $10^{-15}$. As for the linear and stiff part of {\bf EO1} and {\bf EO2} ((\ref{impl1eo}) and (\ref{impl2eo}) respectively), we have considered subroutine ode15s with the same tolerances. On the other hand, in this section, for the implementation of the equation on $W_{h,n}$ in {\bf ACR1} and the first and last equation in {\bf ACR2} ((\ref{impl1acr}) and (\ref{impl2acr}) respectively), as $A_{h,0}$ is a full but small matrix and we consider $k$ as fixed during the whole integration, we have calculated once and for all the matrices $e^{k A_{h,0}}$, $\varphi_1(k A_{h,0})C_h$ and  $\varphi_2(k A_{h,0})C_h$.
The calculation of these matrices is not included in the measured computational time since this initial cost is negligible when integrating till large times but may be very big when $T$ is small. (For another implementation of those terms with no initial cost, look at Section \ref{df}.)

In a first place, we have considered the following one-dimensional Dirichlet boundary value problem whose exact solution is $u(x,t)=e^{t+x^3}$:
\begin{eqnarray}
&&u_t(x,t)=u_{xx}(x,t)+u^2-e^{t+x^3}(9 x^4+ 6 x+e^{t+x^3}-1), \quad 0\le x \le 1, \nonumber \\
&&u(x,0)=e^{x^3}, \nonumber \\
&&u(0,t)=e^t, \quad u(1,t)=e^{t+1}, \quad  t\in[0,0.2]. \label{1dimp}
\end{eqnarray}
For the  spectral space discretization, $16$ Gauss-Lobatto interior nodes have been used so that the error in space is negligible. We also notice that the matrix  $e^{k A_{h,0}}$ has
dimension $16\times 16$ while $\varphi_1(k A_{h,0})C_h$ and  $\varphi_2(k A_{h,0})C_h$ have dimension $16\times 2$ since they are just multiplied by the values at the boundary. Therefore, the calculation of the terms in $e^{k A_{h,0}}$ is more expensive than the calculation of those in
$\varphi_i(k A_{h,0})C_h$ ($i=1,2$). On the other hand, as it is a one-dimensional problem, the function $q$ in {\bf EO1} and {\bf EO2} is calculated directly for every value of $t$
as the straight line which joins the corresponding values $f(t,0,e^t)$ and $f(t,1,e^{t+1})$ at $x=0$ and $x=1$ respectively.
The values of the time stepsize which have been displayed have been $k=10^{-3}, 5\times 10^{-4}, 2.5 \times 10^{-4}, 1.25 \times 10^{-4}, 6.25 \times 10^{-5}$ for {\bf EO} and
$k=10^{-3}, 5\times 10^{-4}, 2.5 \times 10^{-4}, 1.25 \times 10^{-4}, 6.25 \times 10^{-5}, 3.125 \times 10^{-5}$ for {\bf ACR}. The results in terms of maximum error against computational cost are in Figure \ref{f1} and we can see
that techniques \emph{ {\bf ACR1} and {\bf ACR2} are more efficient than  {\bf EO1} and  {\bf EO2}}, that
{\bf ACR2} is three times more efficient than {\bf ACR1} and that {\bf EO1} and {\bf EO2} are very similar in efficiency. In any case, we can check that, as already remarked in
the previous section, for a fixed value of $k$,  {\bf ACR2} is cheaper than {\bf ACR1} and {\bf EO1} cheaper than {\bf EO2}. Moreover we can observe that, at least for this particular problem, for fixed $k$, the error is smaller with the second implementation of both techniques than with the first.

\begin{figure}
\begin{center}
\includegraphics[height=3in,width=4in]{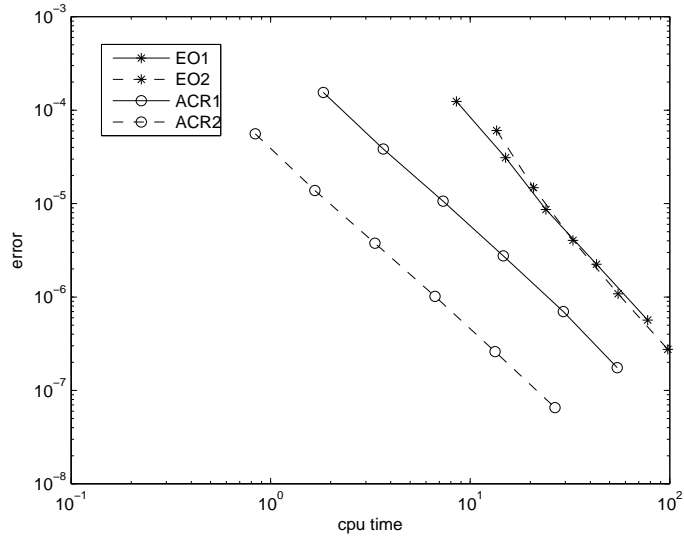}
\end{center}
\caption{Numerical comparison with spectral collocation methods in space and high accuracy in time for the 1-dimensional problem (\ref{1dimp})} \label{f1}
\end{figure}

\begin{figure}
\begin{center}
\includegraphics[height=3in,width=4in]{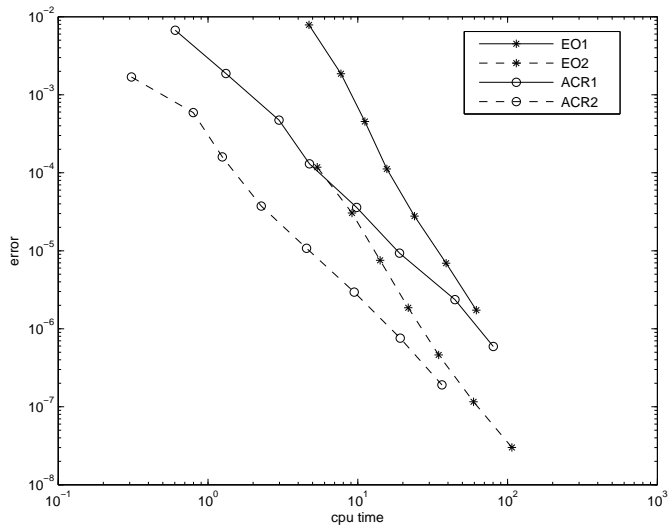}
\end{center}
\caption{Numerical comparison with spectral collocation methods in space and high accuracy in time for the 2-dimensional problem (\ref{2dimp})} \label{f2}
\end{figure}

In a second place, we have considered the two-dimensional problem
\begin{eqnarray}
&&u_t(x,y,t)=u_{xx}(x,y,t)+u_{yy}(x,y,t)+f(t,x,y,u(x,y,t)), \quad 0\le x,y \le 1, \nonumber \\
&&u(x,y,0)=e^{x^3+y^3}, \nonumber \\
&&u(0,y,t)=e^{t+y^3}, \quad u(1,y,t)=e^{t+1+y^3}, \nonumber \\
&&u(x,0,t)=e^{t+x^3}, \quad u(x,1,t)=e^{t+1+x^3}, \quad t\in[0,0.2]. \label{2dimp}
\end{eqnarray}
where $f(t,x,y,u)=u^2 -
e^{t+x^3+y^3}(9(x^4+y^4)+6(x+y)+e^{t+x^3+y^3}-1)$, so that the exact solution is $u(x,y,t)=e^{t+x^3+y^3}$.
Now, $16$ interior Gauss-Lobatto nodes have been taken in each direction of the square for the space discretization and the implementation has been performed with similar remarks
to those of the one-dimensional case. The only remarkable difference is that the function $q(t,x,y)$ in {\bf EO1} and {\bf EO2} must be chosen in a different way. We consider a
function  of the form $q(t,x,y)=r(t,x) f(t,1,y,e^{t+1+y^3})+s(t,x) f(t,0,y,e^{t+y^3})$ which satisfies the corresponding conditions at the boundary and that is achieved if $r(t,x)$ and
$s(t,x)$ satisfy
\begin{eqnarray}
\left( \begin{array}{cc} f(t,1,0,e^{t+1}) & f(t,0,0,e^{t}) \\
                         f(t,1,1,e^{t+2}) & f(t,0,1,e^{t+1}) \end{array} \right) \left( \begin{array}{cc} r(t,x) \\ s(t,x) \end{array} \right)=\left( \begin{array}{c}
                         f(t,x,0,e^{t+x^3}) \\ f(t,x,1,e^{t+1+x^3})\end{array} \right).
                         \nonumber
\end{eqnarray}
(Notice that this technique to calculate $q$ analytically can be applied in a rectangular domain but not in more complicated domains in two dimensions.)
In Figure \ref{f2}, which corresponds to the following values of the timestepsize $k=2 \times 10^{-2}, 10^{-2}, 5 \times 10^{-3}, 2.5  \times 10^{-3},
1.25 \times 10^{-3}, 6.25 \times 10^{-4}, 3.125 \times 10^{-4}$ for {\bf EO} and $k=2.5  \times 10^{-3},
1.25 \times 10^{-3}, 6.25 \times 10^{-4}, 3.125 \times 10^{-4}, 1.5625 \times 10^{-4}, 7.8125 \times 10^{-5}, 3.9063 \times 10^{-5}, 1.9531 \times 10^{-5}$ for {\bf ACR}, we can  see that
the second implementation of both techniques is cheaper than the first and that \emph{the best of all implementations is
 {\bf ACR2}}, at least for a range of errors $\ge 10^{-7}$.

\section{Numerical comparison with finite difference methods in space and middle accuracy in time}
\label{df}

\begin{figure}
\begin{center}
\includegraphics[height=3in,width=4in]{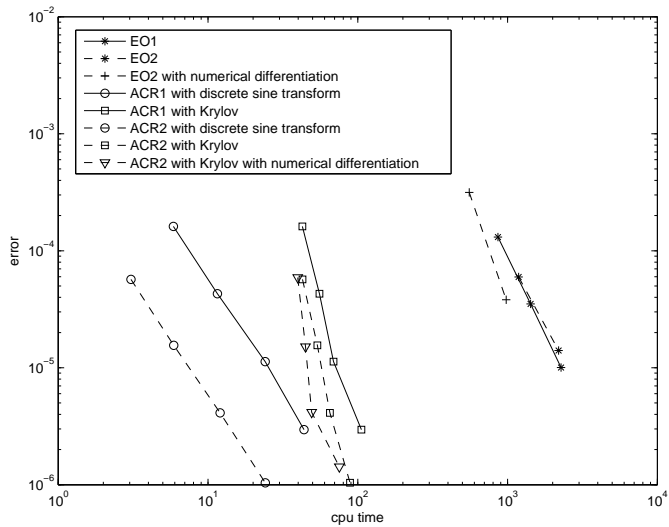}
\end{center}
\caption{Numerical comparison with finite difference methods in space and middle accuracy in time for the 1-dimensional problem (\ref{1dimp})} \label{f3}
\end{figure}

\begin{figure}
\begin{center}
\includegraphics[height=3in,width=4in]{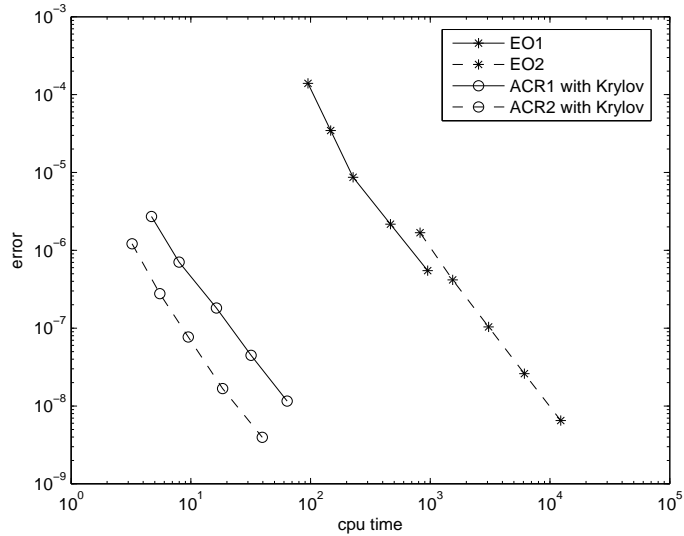}
\end{center}
\caption{Numerical comparison with finite difference methods in space and middle accuracy in time for the 2-dimensional problem (\ref{2dimpb})} \label{f4}
\end{figure}

In this section we have been a bit less demanding when solving each part of the splitting.  We have just considered $10^{-7}$ and
$10^{-8}$ as relative and absolute tolerances respectively for the standard subroutines ode45 and ode15s. As Strang method just has second-order accuracy,
it is usually used for problems in which a very high precision is not required.
Moreover, in space we have considered finite differences of just second
order accuracy in the space grid. More particularly, as in the problems above the operator $A$ is the Laplacian, we have taken the standard symmetric second-order difference scheme
in $1$ dimension and the five-point formula in $2$ dimensions \cite{S}. We have considered as space grid $h=5 \times 10^{-4}$ for the $1$-dimensional case and $h=2 \times 10^{-2}$
for the $2$-dimensional case. With this type of implementation, the
matrix $A_{h,0}$ is sparse and, in this particular case, their eigenvalues and eigenvectors are well-known \cite{I}. Because of the former, it is natural to use standard Krylov
subroutines \cite{N} in {\bf ACR1} and {\bf ACR2} to calculate the application of exponential-type functions over vectors. Due to the latter, which is more specific of this
particular example and space discretization, in order to calculate the same terms, it seems  advantageous to use the discrete sine transform in the same way that FFT is used in
Poisson solvers \cite{I}. When using Krylov subroutines \cite{N}, we have considered the default tolerance $10^{-7}$.  The comparison is performed in Figure \ref{f3}
for the
1-dimensional problem (\ref{1dimp}) with $k=10^{-3}, 5 \times 10^{-4}, 2.5 \times 10^{-4}$ for {\bf EO1},
$k=10^{-3}, 5 \times 10^{-4}$ for {\bf EO2} and $k=10^{-3}, 5 \times 10^{-4}, 2.5 \times 10^{-4}, 1.25 \times 10^{-4}$ for {\bf ACR}.
\emph{We again see that {\bf ACR1} and  {\bf ACR2} are more competitive than  {\bf EO1} and  {\bf EO2}}. Although, for a fixed value of $k$, {\bf EO2} takes more computational time than {\bf EO1}, in the end they are very similar in efficiency since, at least in this case, the error is also quite smaller. As for {\bf ACR1} and {\bf ACR2}, {\bf ACR2} is more competitive since not only the computational time is smaller for a fixed value of $k$  but also the error is smaller. In this particular case, considering discrete sine transforms is much cheaper than using Krylov techniques. However, for a general operator $A$, that may not be possible and that is why it is also interesting to see the comparison when using these techniques. In any case, \emph{the worst of {\bf ACR} implementations is about 20  times cheaper than the best of  {\bf EO}}.

Moreover,  following \cite{EO3}, we have also considered numerical differentiation in order to try to get local order $3$ with {\bf EO1} and {\bf EO2} in (\ref{1dimp}).
More precisely, theoretically, a function $q$ should be taken for which $\partial q(t)= \partial f(t,u(t))$ and $\partial A q(t)=\partial A f(t,u(t)).$
Although, even if we were able to construct that function, the order for the global error does not improve, it is interesting to see whether the fact that the local errors maybe smaller implies a better overall behaviour.
Notice that, in (\ref{1dimp}),
\begin{eqnarray}
\frac{d}{dx^2} f=f_{xx}+ 2 f_{x,u} u_x+ f_{uu} u_x^2 + f_u u_{xx}.
\label{d2f}
\end{eqnarray}
As $\partial u_{xx}=g'(t)-\partial f(t,u)$, numerical differentiation is just required to calculate $\partial u_x$. For that, we have considered the second-order scheme
\begin{eqnarray*}
 u_x(0,t)&\approx& \frac{-\frac{3}{2}u(0,t)+2 u(h,t)-\frac{1}{2}u(2h,t)}{h}, \\
 u_x(1,t)&\approx& \frac{\frac{3}{2}u(1,t)-2 u(1-h,t)+\frac{1}{2}u(1-2h,t)}{h}.
 \end{eqnarray*}
>From a theoretical point of view, to achieve local order $3$, at each step we would need these derivatives at any continuous time $t\in [t_n,t_{n+1})$. However, we just have approximations for the interior values $u(h,t)$, $u(2h,t)$, $u(1-h,t)$, $u(1-2h,t)$ at time $t_n$. Because of this, in formula (\ref{d2f}), we have evaluated all terms at continuous $t$ except for the term $u_x$, which is just approximated at $t=t_n$. In such a way, although not shown here for the sake of brevity, the local error shows order a bit less than $3$ but higher than $2.5$. Besides, although numerical differentiation is a badly-posed problem, its effect is still not visible with the considered value of $h$ for the first derivative and the range of errors which we are considering. The results for {\bf EO2} with numerical differentiation are shown in Figure \ref{f3} for $k=2\times 10^{-3}$ and $k=10^{-3}$. We can see that, \emph{in terms of computational efficiency, numerical differentiation is slightly worth doing}.

As for {\bf ACR2}, considering also terms of second order in $s$ for the boundaries of the problems in which the operator $A$  appears,
the following full scheme turns up (see the appendix):
\begin{eqnarray}
W_{h,n}(\frac{k}{2})&=&e^{\frac{k}{2} A_{h,0}} U_{h}^n+ \frac{k}{2} \varphi_1( \frac{k}{2} A_{h,0})C_h g(t_n)+\frac{k^2}{4} \varphi_2( \frac{k}{2} A_{h,0})C_h\partial
       Au(t_n)
\nonumber \\
&&+\frac{k^3}{8} \varphi_3(\frac{k}{2}A_{h,0}) C_h \partial A^2 u(t_n)\nonumber \\
V_h^n &=& \Psi_{k}^{f,t_n}(W_{h,n}(\frac{k}{2})), \nonumber \\
U_h^{n+1}&=&e^{\frac{k}{2} A_{h,0}}V_h^n
\nonumber \\
&&+ \frac{k}{2} \varphi_1( \frac{k}{2} A_{h,0})C_h \partial [u(t_n)+k(\frac{1}{2}  A u(t_n)+ f(t_n,u(t_n))
\nonumber \\
&& \hspace{1cm}
+k^2(\frac{1}{8} A^2 u(t_n)+\frac{1}{2} f_u(t_n,u(t_n)) A u(t_n)  \nonumber \\
&& \hspace{1cm}
+\frac{1}{2}(f_t(t_n,u(t_n))+f_u(t_n,u(t_n))f(t_n,u(t_n)))] \nonumber \\
&& + \frac{k^2}{4} \varphi_2( \frac{k}{2} A_{h,0}) C_h \partial [A u(t_n)+\frac{k}{2} A^2 u(t_n)+k A f(t_n,u(t_n))]
\nonumber \\
&&+
\frac{k^3}{8} \varphi_3( \frac{k}{2} A_{h,0}) C_h \partial A^2 u(t_n).
\label{acr2dn}
\end{eqnarray}
As $\partial A^2 u=\partial A\dot{u}- \partial A f=\ddot{g}-\partial (f_t+f_u \dot{u})-\partial A f$, what is again necessary is to approximate $u_x$ with numerical differentiation and we have done it in the same way as before.
As it is observed in Figure \ref{f3}, \emph{there is a small ganancy in efficiency when using numerical differentiation with
{\bf ACR2} although it is not extremely significant}.

Notice that for both {\bf EO2} and {\bf ACR2}, for a fixed value of $k$, the computational cost does not increase but is slightly smaller when using numerical differentiation. This must be due to the fact that the standard subroutines which are used converge more quickly when numerical differentiation is applied. A full explanation for that is out of the scope of this paper although it might be a subject of future research.

Let us now see what happens with a bidimensional problem. In order to assure that the errors in space are negligible without having to decrease too much the space grid, we have considered
\begin{eqnarray}
&&u_t(x,y,t)=u_{xx}(x,y,t)+u_{yy}(x,y,t)+f(t,x,y,u(x,y,t)), \quad 0\le x,y \le 1, \nonumber \\
&&u(x,y,0)=x^2+y^2, \nonumber \\
&&u(0,y,t)=e^{t}y^2, \quad u(1,y,t)=e^{t}(1+y^2), \nonumber \\
&&u(x,0,t)=e^{t}x^2, \quad u(x,1,t)=e^{t}(1+x^2), \quad t\in[0,0.2]. \label{2dimpb}
\end{eqnarray}
where $f(t,x,y,u)=u^2 -
e^{2t}(x^2+y^2)^2+e^t(x^2+y^2-4)$, so that the exact solution is $u(x,y,t)=e^{t}(x^2+y^2)$.
We have implemented {\bf EO1} and {\bf EO2} calculating $q$ in a similar way as in the bidimensional problem of the previous section and {\bf ACR1} and {\bf ACR2} again with Krylov subroutines \cite{N}. In Figure \ref{f4} we have displayed the results corresponding to {\bf EO1} and {\bf EO2} with  $k=10^{-2}, 5 \times 10^{-3}, 2.5 \times 10^{-3}, 1.25 \times 10^{-3}, 6.25 \times 10^{-4}$ and to {\bf ACR1} and {\bf ACR2} with $k= 1.25 \times 10^{-3}, 6.25 \times 10^{-4}, 3.125 \times 10^{-4}, 1.5625 \times 10^{-4}, 7.8125 \times 10^{-5}$. We can see that, in this problem, the second implementation is the most efficient for {\bf ACR} and the first is the best for {\bf EO}. Moreover, \emph{{\bf ACR2} is about  300 times more efficient than {\bf EO1}}.

Considering numerical differentiation in two dimensions is also possible but we would like to remark that, with {\bf EO} techniques, that it is not as plausible as in one dimension since, apart from approximating numerically $\partial A f(t,u(t))$ at each step and calculating a function $\tilde{q}(t)$ which coincides with it at the boundary, a function $q(t)$ must be calculated such that
\begin{eqnarray}
Aq(t)&=&\tilde{q}(t), \nonumber \\
\partial q(t) &=&\partial f(t,u(t)).
\label{elliptic}
\end{eqnarray}
In one dimension, this was achieved just by integrating twice the linear function $\tilde{q}(t)$ and that was done analytically for every value of $t\in [t_n,t_{n+1})$. However, in two dimensions, that cannot be done any more and the elliptic problems (\ref{elliptic}) should be numerically solved, not only for every value $t_n$, but even theoretically for every $t\in [t_n,t_{n+1})$.
In contrast, notice that numerical differentiation with ACR (\ref{acr2dn}) just requires approximating $\partial A f(t_n,u(t_n))$ at each step and no elliptic problem must be numerically solved at continuous time.
Besides, with respect to the same method but without numerical differentation, the additional cost mainly consists of just two more terms per step which contain $\varphi_3(\frac{k}{2} A_{h,0})$.
In any case, we do not either include numerical differentiation with {\bf ACR} here for the sake of clarity and brevity.

%\newpage

\section*{Acknowledgements}

This work has been  supported by project MTM 2015-66837-P.

%
%  Bibliography. Follow the usual conventions.
%
\section*{Appendix: Analysis of ACR2 with and without numerical differentiation}
In this appendix, we state where formula (\ref{impl2acr}) comes from for ACR2 implementation and justify that the local and global error behaves with second order of accuracy.
Moreover, we also state where formula (\ref{acr2dn}) comes from for ACR2 implementation with numerical differentation and justify that the local error behaves with third order of accuracy in such a case. We concentrate here on the results for the local errors after time semidiscretization since the results for
the errors after full discretization would follow in the same way than in \cite{ACRnl}. As in \cite{ACRnl}, the restriction of the operator $A$ to the domain with
vanishing boundary is denoted by $A_0$, which is a generator of a $C_0$-semigroup which is denoted by $e^{t A_0}$.

The problems to be solved after time semidiscretization are
\begin{eqnarray}
\left\{ \begin{array}{rcl} w_{n,1}'(s)&=& A w_{n,1}(s),  \\
w_{n,1}(0)&=& u_n,  \\
\partial w_{n,1}(s)&=& \partial \hat{w}_{n,1}(s),
\end{array} \right.
 \label{wn1} \\
 %\left\{ \begin{array}{rcl} v_n'(s)&=& f(t_n+s,v_n(s)) \\
 %v_n(0)&=& w_{n,1}(\frac{k}{2}), \end{array} \right.  \\
 \left\{ \begin{array}{rcl} w_{n,2}'(s)&=& A w_{n,2}(s),  \\
w_{n,2}(0)&=& \Psi_k^{f,t_n} (w_{n,1}(\frac{k}{2})),  \\
\partial w_{n,2}(s)&=& \partial \hat{w}_{n,2}(s),
\end{array} \right.
 \label{wn2}
\end{eqnarray}
with
\begin{eqnarray}
\hat{w}_{n,1}(s)&=&u(t_n)+s A u(t_n), \label{what1} \\
\hat{w}_{n,2}(s)&=&u(t_n)+\frac{k}{2} A u(t_n)+k f(t_n,u(t_n))+s A u(t_n). \label{what2}
\end{eqnarray}
Then,
$$u^{n+1}=w_{n,2}(\frac{k}{2}),$$
and the following result follows.
\begin{theorem}
Under the same hypotheses of Theorem 2 in \cite{ACRnl}, when integrating (\ref{ibvp}) with Strang method using the technique (\ref{wn1})-(\ref{wn2}) with $\hat{w}_{n,1}$ in (\ref{what1}) and $\hat{w}_{n,2}$ in (\ref{what2}), the local error $\rho_{n+1}$ satisfies $\rho_{n+1}=O(k^2)$.
\end{theorem}
\begin{proof}
By definition, $\rho_{n+1}=\bar{u}_{n+1}-u(t_{n+1})$, where $\bar{u}_{n+1}$ is calculated through $\bar{w}_{n,1}$ and $\bar{w}_{n,2}$  as in (\ref{wn1}) and (\ref{wn2}) but substituting $u_n$ by $u(t_n)$. Then,
\begin{eqnarray}
\bar{w}_{n,1}'(s)-\hat{w}_{n,1}'(s)&=&A (\bar{w}_{n,1}(s)-\hat{w}_{n,1}(s))+s A^2 u(t_n), \nonumber \\
\bar{w}_{n,1}(0)-\hat{w}_{n,1}(0)&=&0, \nonumber \\
\partial[ \bar{w}_{n,1}(s)-\hat{w}_{n,1}(s)]&=&0,
\nonumber
\end{eqnarray}
from what, using the variation of constants,
$$
\bar{w}_{n,1}(\frac{k}{2})-\hat{w}_{n,1}(\frac{k}{2})=\int_0^{\frac{k}{2}}e^{(\frac{k}{2}-\tau)A_0} \tau A^2 u(t_n) d \tau= \frac{k^2}{4}\varphi_2(\frac{k}{2} A_0) A^2 u(t_n).
$$
On the other hand,
\begin{eqnarray}
\bar{w}_{n,2}'(s)-\hat{w}_{n,2}'(s)&=&A (\bar{w}_{n,2}(s)-\hat{w}_{n,2}(s))  \nonumber \\
&& +\frac{k}{2} A^2 u(t_n) + k A f(t_n,u(t_n))+s A^2 u(t_n), \nonumber \\
\bar{w}_{n,2}(0)-\hat{w}_{n,2}(0)&=&\Psi_k^{f,t_n}(\bar{w}_{n,1}(\frac{k}{2}))-u(t_n)-\frac{k}{2}A u(t_n)-k f(t_n,u(t_n)), \nonumber \\
\partial[ \bar{w}_{n,2}(s)-\hat{w}_{n,2}(s)]&=&0.
\nonumber
\end{eqnarray}
Therefore, also by the variation of constants formula,
\begin{eqnarray}
\lefteqn{\bar{w}_{n,2}(\frac{k}{2})-\hat{w}_{n,2}(\frac{k}{2})=e^{\frac{k}{2}A_0}[\Psi_k^{f,t_n}(\bar{w}_{n,1}(\frac{k}{2}))-u(t_n)-\frac{k}{2} A u(t_n)-k f(t_n,u(t_n))]} \nonumber \\
&&+\frac{k^2}{4}\varphi_1(\frac{k}{2} A_0) A^2 u(t_n)+\frac{k^2}{2}\varphi_1(\frac{k}{2} A_0) A f(t_n,u(t_n))+\frac{k^2}{4}\varphi_2(\frac{k}{2} A_0) A^2 u(t_n).
\nonumber
\end{eqnarray}
>From this, using (\ref{what1}) and (\ref{what2}) and Taylor expansions,
$$\rho_{n+1}=\bar{w}_{n,2}(\frac{k}{2})-u(t_{n+1})=O(k^2).$$
\end{proof}
In order to be able to apply a summation-by-parts argument, so that order $2$ is also proved for the global error, the following result is necessary, which assumes a bit more regularity on the solution of the problem and a bit more of accuracy on the integrator $\Psi_k$ (see \cite{ACRnl} for more details).
\begin{theorem}
Under the same hypotheses of Theorem 3 in \cite{ACRnl}, when integrating (\ref{ibvp}) with Strang method using the technique (\ref{wn1})-(\ref{wn2}) with $\hat{w}_{n,1}$ in (\ref{what1}) and $\hat{w}_{n,2}$ in (\ref{what2}), the local error $\rho_{n+1}$ satisfies $A_0^{-1} \rho_{n+1}=O(k^3)$.
\end{theorem}
\begin{proof}
It suffices to notice that the terms in $k^2$ in the previous expression of $\rho_{n+1}$ can also be written as
\begin{eqnarray}
\lefteqn{e^{\frac{k}{2} A_0}\bigg[\frac{k^2}{4} \varphi_2(\frac{k}{2} A_0) A^2 u(t_n)+\frac{k^2}{2} f_u(t_n,u(t_n))A u(t_n)}  \nonumber \\
&& \hspace{0.5cm}+\frac{k^2}{2}[f_t(t_n,u(t_n))+f_u(t_n,u(t_n))f(t_n,u(t_n))]\bigg] \nonumber \\
&&+\frac{k^2}{4} \varphi_1(\frac{k}{2} A_0) A^2 u(t_n)+\frac{k^2}{2} \varphi_1(\frac{k}{2} A_0) A f(t_n,u(t_n))  \nonumber \\
&&+\frac{k^2}{4} \varphi_2(\frac{k}{2} A_0) A^2 u(t_n)-\frac{k^2}{2}u''(t_n). \nonumber
\end{eqnarray}
Then, using that, because of the definition of $\varphi_j$ \cite{ACRnl},
\begin{eqnarray*}
 A_0^{-1} e^{\frac{k}{2} A_0}&=&A_0^{-1}+\frac{k}{2}\varphi_1(\frac{k}{2} A_0), \quad A_0^{-1} \varphi_1(\frac{k}{2} A_0)= A_0^{-1}+\frac{k}{2}\varphi_2(\frac{k}{2} A_0), \\
 A_0^{-1} \varphi_2(\frac{k}{2} A_0)&=& \frac{1}{2}A_0^{-1}+\frac{k}{2}\varphi_3(\frac{k}{2} A_0),
\end{eqnarray*}
the following is deduced simplifying the notation,
$$A_0^{-1} \rho_{n+1}=\frac{k^2}{2} A_0^{-1}[A^2 u+ f_u A u+ A f+f_t+f_u f-u'']+O(k^3)=O(k^3).$$

\end{proof}

With numerical differentiation, the problems to be solved after time semidiscretization are those in (\ref{wn1}) and (\ref{wn2}), but with
\begin{eqnarray}
\hat{w}_{n,1}(s)&=&u(t_n)+s A u(t_n)+\frac{s^2}{2} A^2 u(t_n), \label{what1nd} \\
\hat{w}_{n,2}(s)&=&u(t_n)+\frac{k}{2} A u(t_n)+\frac{k^2}{8} A^2 u(t_n)+k f(t_n,u(t_n)) \nonumber \\
&&+\frac{k^2}{2} [f_u(t_n,u(t_n)) A u(t_n) +f_t(t_n,u(t_n))+f_u(t_n,u(t_n)) f(t_n,u(t_n))] \nonumber \\
&&+s A u(t_n)+\frac{sk}{2} A^2 u(t_n) + sk A f(t_n,u(t_n))+\frac{s^2}{2} A^2 u(t_n)]. \label{what2nd}
\end{eqnarray}
Then, we have the following result for the local error which implies, through the standard argument of convergence which was used in \cite{ACRnl} for Lie-Trotter, that the global error for the full discretization behaves with order $2$ in the timestepsize.
\begin{theorem}
Under the same hypotheses of Theorem 3 in \cite{ACRnl} and assuming also that $u(t)\in D(A^3)$ for $t\in [0,T]$ and $A^3 u\in C([0,T], X)$, when integrating (\ref{ibvp}) with Strang method using the technique (\ref{wn1})-(\ref{wn2}) with $\hat{w}_{n,1}$ in (\ref{what1nd}) and $\hat{w}_{n,2}$ in (\ref{what2nd}), the local error $\rho_{n+1}$ satisfies $\rho_{n+1}=O(k^3)$.
\end{theorem}
\begin{proof}
We notice that now
\begin{eqnarray}
\bar{w}_{n,1}'(s)-\hat{w}_{n,1}'(s)&=&A (\bar{w}_{n,1}(s)-\hat{w}_{n,1}(s))+\frac{s^2}{2} A^3 u(t_n), \nonumber \\
\bar{w}_{n,1}(0)-\hat{w}_{n,1}(0)&=&0, \nonumber \\
\partial[ \bar{w}_{n,1}(s)-\hat{w}_{n,1}(s)]&=&0,
\nonumber
\end{eqnarray}
Therefore, by the variation of constants formula,
$$
\bar{w}_{n,1}(\frac{k}{2})-\hat{w}_{n,1}(\frac{k}{2})=\int_0^{\frac{k}{2}}e^{(\frac{k}{2}-\tau)A_0} \frac{\tau^2}{2} A^3 u(t_n) d \tau= \frac{k^3}{8}\varphi_3(\frac{k}{2} A_0) A^3 u(t_n).
$$
On the other hand, simplifying the notation,
\begin{eqnarray}
\bar{w}_{n,2}'(s)-\hat{w}_{n,2}'(s)&=&A (\bar{w}_{n,2}(s)-\hat{w}_{n,2}(s))+\frac{k^2}{8} A^3 u + \frac{k^2}{2} A f_u A u \nonumber \\
&&+\frac{k^2}{2}A (f_t+f_u f)+\frac{sk}{2} A^3 u+ s k A f+\frac{s^2}{2}A^3 u,
 \nonumber \\
\bar{w}_{n,2}(0)-\hat{w}_{n,2}(0)&=&\Psi_k^{f,t_n}(\bar{w}_{n,1}(\frac{k}{2})) \nonumber \\
&& -[u+\frac{k}{2}A u(t_n)+\frac{k^2}{8} A^2 u +k f+\frac{k^2}{2}f_u Au+\frac{k^2}{2}(f_t+f_u f)], \nonumber \\
\partial[ \bar{w}_{n,2}(s)-\hat{w}_{n,2}(s)]&=&0,
\nonumber
\end{eqnarray}
from what, also by the variation of constants formula,
\begin{eqnarray}
\lefteqn{\bar{w}_{n,2}(\frac{k}{2})-\hat{w}_{n,2}(\frac{k}{2})=e^{\frac{k}{2}A_0}\bigg[\Psi_k^{f,t_n}(\bar{w}_{n,1}(\frac{k}{2}))} \nonumber \\
&& \hspace{1cm} -[u+\frac{k}{2}A u(t_n)+\frac{k^2}{8} A^2 u +k f+\frac{k^2}{2}f_u Au+\frac{k^2}{2}(f_t+f_u f)]\bigg] \nonumber \\
&&+\int_0^{\frac{k}{2}} e^{(\frac{k}{2}-\tau)A_0}[\frac{k^2}{8} A^3 u + \frac{k^2}{2} A f_u A u +\frac{k^2}{2}A (f_t+f_u f)+\frac{\tau k}{2} A^3 u+ \tau k A f+\frac{\tau^2}{2}A^3 u] d\tau  \nonumber \\
&&=O(k^3),
\nonumber
\end{eqnarray}
and therefore
\begin{eqnarray}
\bar{w}_{n,2}(\frac{k}{2})&=&u+k(Au +f)+\frac{k^2}{2}(A^2 u+Af+f_u Au+ f_t+f_u f)+O(k^3)
\nonumber \\
&=&u+k \dot{u}+ \frac{k^2}{2} \ddot{u}=u(t_{n+1})+O(k^3).
\nonumber
\end{eqnarray}

\end{proof}

\end{document}